\documentclass{amsart}

\usepackage{cite}

\newtheorem{theorem}{Theorem}[section]
\newtheorem{example}[theorem]{Example}
\theoremstyle{definition}
\newtheorem{definition}[theorem]{Definition}
\theoremstyle{remark}
\newtheorem{remark}[theorem]{Remark}

\numberwithin{equation}{section}

\begin{document}

\title[Solving Abel integral equations of first kind]{Solving
Abel integral equations of first kind\\ via fractional calculus}

\author[S. Jahanshahi]{S. Jahanshahi} 
\address{Department of Mathematics, Science and Research Branch,
Islamic Azad University, Tehran, Iran}
\email{s.jahanshahi@iausr.ac.ir}

\author[E. Babolian]{E. Babolian} 
\address{Department of Mathematics, Science and Research Branch,
Islamic Azad University, Tehran, Iran}
\email{Babolian@khu.ac.ir}

\author[D. F. M. Torres]{D. F. M. Torres} 
\address{Center for Research and Development in Mathematics and Applications (CIDMA),
Department of Mathematics, University of Aveiro, 3810--193 Aveiro, Portugal}
\email[Corresponding author]{delfim@ua.pt}

\author[A. R. Vahidi]{A. R. Vahidi} 
\address{Department of Mathematics, Shahre Rey Branch,
Islamic Azad University, Tehran, Iran}
\email{alrevahidi@iausr.ac.ir}


\thanks{This is a preprint of a paper whose final and definite form is 
in Journal of King Saud University (Science), ISSN 1018-3647. 
Paper submitted 30/Aug/2014; revised 25/Sept/2014; accepted for publication 29/Sept/2014.
Corresponding author: Delfim F. M. Torres (delfim@ua.pt),
Department of Mathematics, University of Aveiro, 3810--193 Aveiro, Portugal.
Tel: +351 234370668, Fax: +351 234370066.}

\subjclass[2010]{26A33; 45E10}

\keywords{Abel integral equation; singular integral equations;
Caputo fractional derivatives; fractional integrals}


\begin{abstract}
We give a new method for numerically solving Abel integral equations of first kind.
An estimation for the error is obtained. The method is based on approximations
of fractional integrals and Caputo derivatives.
Using trapezoidal rule and Computer Algebra System Maple,
the exact and approximation values of three Abel integral equations are found,
illustrating the effectiveness of the proposed approach.
\end{abstract}


\maketitle


\section{Introduction}

Consider the following generalized Abel integral equation of first kind:
\begin{equation}
\label{eq:1.1}
f(x)=\int_0^x\frac{k(x,s)g(s)}{(x-s)^\alpha}ds,
\quad 0<\alpha<1, \quad 0 \leq x \leq b,
\end{equation}
where $g$ is the unknown function to be found, $f$ is a well behaved function,
and $k$ is the kernel. This equation is one of the most famous equations
that frequently appears in many physical and engineering problems, like semi-conductors,
heat conduction, metallurgy and chemical reactions \cite{bib:10,bib:11}.
In experimental physics, Abel's integral equation of first kind \eqref{eq:1.1} finds
applications in plasma diagnostics, physical electronics, nuclear physics,
optics and astrophysics \cite{Kosarev,MR1230950}. To determine the radial distribution of
the radiation intensity of a cylinder discharge in plasma physics, for example,
one needs to solve an integral equation \eqref{eq:1.1} with $\alpha = \frac{1}{2}$.
Another example of application appears when one
describes velocity laws of stellar winds \cite{MR1230950}.
If $k(x,s)=\frac{1}{\Gamma(1-\alpha)}$, then \eqref{eq:1.1} is a fractional integral
equation of order $1-\alpha$ \cite{bib:23}. This problem is a generalization
of the tautochrone problem of the calculus of variations, and is related
with the born of fractional mechanics \cite{MR1438729}.
The literature on integrals and derivatives of fractional order is now vast
and evolving: see, e.g., \cite{bib:09,bib:08,Tarasov:13,TM:K:M:11,Wang:Zhou:11}.
The reader interested in the early literature, showing that Abel's integral
equations may be solved with fractional calculus, is referred to \cite{bib:Rev1:1}.
For a concise and recent discussion on the solutions of Abel's integral equations
using fractional calculus see \cite{bib:Rev1:2}.

Many numerical methods for solving \eqref{eq:1.1} have been developed
over the past few years,
such as product integration methods \cite{bib:03,bib:04},
collocation methods \cite{bib:06},
fractional multi step methods \cite{bib:16,bib:15,bib:22},
backward Euler methods \cite{bib:03},
and methods based on wavelets \cite{MR2541683,MR2736623,MR2883895}.
Some semi analytic methods, like the Adomian decomposition method,
are also available, which produce a series solution \cite{bib:05}.
Unfortunately, the Abel integral equation \eqref{eq:1.1} is an ill-posed problem.
For $k(x,s)=\frac{1}{\Gamma(1-\alpha)}$, Gorenflo \cite{bib:10} presented
some numerical methods based on fractional calculus, e.g., using
the Grunwald--Letnikov difference approximation
\begin{equation}
\label{eq:1.2}
D^\alpha f \simeq h^{-\alpha} \sum_{r=0}^{n} (-1)^r \binom{\alpha}{r} f(x-rh).
\end{equation}
If $f$ is sufficiently smooth and vanishes at $x\leq 0$,
then formula \eqref{eq:1.2} has accuracy of order $O(h^2)$,
otherwise it has accuracy of order $O(h)$. On the other hand,
Lubich \cite{bib:16,bib:15} introduced a fractional multi-step method for the Abel integral
equation of first kind, and Plato \cite{bib:22} considered fractional multi-step methods
for weakly singular integral equations of first kind with a perturbed right-hand side.
Liu and Tao solved the fractional integral equation, transforming it into
an Abel integral equation of second kind \cite{bib:18}.
A method based on Chebyshev polynomials is given in \cite{bib:02}.
Here we propose a method to solve an Abel integral equation of first kind
based on a numerical approximation of fractional integrals and Caputo derivatives
of a given function $f$ belonging to $C^n[a,b]$ (see Theorem~\ref{thm:mr}).

The structure of the paper is as follows. In Section~\ref{sec:2} we recall the necessary definitions
of fractional integrals and derivatives and explain some useful relations between them.
Section~\ref{sec:3} reviews some numerical approximations for fractional integrals and derivatives.
The original results are then given in Section~\ref{sec:4}, where we introduce our method to approximate
the solution of the Abel equation at the given nodes and we obtain an upper bound for the error.
In Section~\ref{sec:5} some examples are solved to illustrate the accuracy of the proposed method.


\section{Definitions, Relations and Properties of Fractional Operators}
\label{sec:2}

Fractional calculus is a classical area with many good books available.
We refer the reader to \cite{book:M:Torres:12,bib:23}.

\begin{definition}
\label{def:RL:int}
Let $\alpha>0$ with $n-1<\alpha\leq n$, $n\in\mathbb{N}$, and $a<x<b$.
The left and right Riemann--Liouville fractional integrals
of order $\alpha$ of a given function $f$ are defined by
\begin{equation*}
_aJ_x^\alpha f(x)=\frac{1}{\Gamma(\alpha)}
\int_a^x(x-t)^{\alpha-1}f(t)dt
\end{equation*}
and
\begin{equation*}
_xJ_b^{\alpha}f(x)=\frac{1}{\Gamma(\alpha)}\int_x^b(t-x)^{\alpha-1}f(t)dt,
\end{equation*}
respectively, where $\Gamma$ is Euler's gamma function, that is,
\begin{equation*}
\Gamma(x)=\int_0^\infty t^{x-1}e^{-t}dt.
\end{equation*}
\end{definition}

\begin{definition}
\label{def:RL:der}
The left and right Riemann--Liouville fractional derivatives of order $\alpha>0$,
$n-1<\alpha\leq n$, $n \in\mathbb{N}$, are defined by
\begin{equation*}
_aD_x^{\alpha}f(x)=\frac{1}{\Gamma(n-\alpha)}\frac{d^n}{dx^n}\int_a^x(x-t)^{n-\alpha-1}f(t)dt
\end{equation*}
and
\begin{equation*}
_xD_b^\alpha f(x)=\frac{(-1)^n}{\Gamma(n-\alpha)}\frac{d^n}{dx^n}\int_x^b(t-x)^{n-\alpha-1}f(t)dt,
\end{equation*}
respectively.
\end{definition}

\begin{definition}
\label{def:Caputo}
The left and right Caputo fractional derivatives of order $\alpha>0$,
$n-1<\alpha\leq n$, $n \in\mathbb{N}$, are defined by
\begin{equation*}
_a^CD_x^\alpha f(x)
=\frac{1}{\Gamma(n-\alpha)}\int_a^x(x-t)^{n-\alpha-1}f^{(n)}(t)dt
\end{equation*}
and
\begin{equation*}
_x^CD_b^\alpha f(x)
=\frac{(-1)^n}{\Gamma(n-\alpha)}\int_x^b(t-x)^{n-\alpha-1}f^{(n)}(t)dt,
\end{equation*}
respectively.
\end{definition}

\begin{definition}
\label{def:GL}
Let $\alpha>0$. The Grunwald--Letnikov fractional derivatives are defined by
\begin{equation*}
D^\alpha f(x)=\lim_{h \to 0} h^{-\alpha}\sum_{r=0}^\infty (-1)^r \binom{\alpha}{r} f(x-rh)
\end{equation*}
and
\begin{equation*}
D^{-\alpha} f(x)=\lim_{h \to 0} h^{\alpha}\sum_{r=0}^\infty
\left[
\begin{array}{c}\alpha\\
r
\end {array}
\right]
f(x-rh),
\end{equation*}
where
\begin{equation*}
\left[
\begin{array}{c}
\alpha\\
r
\end{array}
\right]
=\frac{\alpha(\alpha+1)(\alpha+2) \cdots (\alpha+r-1)}{r!}.
\end{equation*}
\end{definition}

\begin{remark}
The Caputo derivatives (Definition~\ref{def:Caputo}) have some advantages
over the Riemann--Liouville derivatives (Definition~\ref{def:RL:der}).
The most well known is related with the Laplace transform method
for solving fractional differential equations. The Laplace transform
of a Riemann--Liouville derivative leads to boundary conditions containing
the limit values of the Riemann--Liouville fractional derivative at the lower terminal $x=a$.
In spite of the fact that such problems can be solved analytically, there is no physical interpretation
for such type of boundary conditions. In contrast, the Laplace transform of a Caputo derivative imposes boundary
conditions involving integer-order derivatives at $x=a$, which usually are acceptable physical conditions.
Another advantage is that the Caputo derivative of a constant function is zero,
whereas for the Riemann--Liouville it is not. For details see \cite{bib:26}.
\end{remark}

\begin{remark}
The Grunwald--Letnikov definition gives a generalization of the ordinary discretization
formulas for derivatives with integer order.
The series in Definition~\ref{def:GL} converge absolutely and uniformly
for each $\alpha>0$ and for every bounded function $f$. The discrete approximations
derived from the Grunwald--Letnikov fractional derivatives, e.g., \eqref{eq:1.2},
present some limitations. First, they frequently originate unstable
numerical methods and henceforth usually a shifted Grunwald--Letnikov
formula is used instead. Another disadvantage is that the order of accuracy
of such approximations is never higher than one. For details see \cite{bib:03}.
\end{remark}

The following relations between Caputo
and Riemann--Liouville fractional derivatives hold \cite{bib:23}:
\begin{equation*}
_a^CD_x^\alpha f(x)={_aD_x^\alpha} f(x)
-\sum_{k=0}^{n-1}\frac{f^{(k)}(a)}{\Gamma(k-\alpha+1)}(x-a)^{k-\alpha}
\end{equation*}
and
\begin{equation*}
_x^CD_b^\alpha f(x)={_xD_b^\alpha} f(x)
-\sum_{k=0}^{n-1}\frac{f^{(k)}(b)}{\Gamma(k-\alpha+1)}(b-x)^{k-\alpha}.
\end{equation*}
Therefore, if $f\in C^n[a,b]$ and $f^{(k)}(a)=0$, $k=0,1,\ldots,n-1$, then
\begin{equation*}
_a^CD_x^\alpha f= {_aD_x^\alpha} f,
\end{equation*}
if $f^{(k)}(b)=0$, $k=0,1,\ldots,n-1$, then
\begin{equation*}
_x^CD_b^\alpha f={_xD_b^\alpha} f.
\end{equation*}
Other useful properties of fractional integrals and derivatives are:
all fractional operators are linear, that is,
if $L$ is an arbitrary fractional operator, then
\begin{equation*}
L(t f+s g)=t L(f)+s L(g)
\end{equation*}
for all functions $f,g\in C^n[a,b]$ or $f,g\in L^p(a,b)$
and $t,s\in \mathbb{R}$; if $\alpha,\beta>0$, then
\begin{equation*}
J^\alpha J^\beta=J^{\alpha+\beta},
\quad D^\alpha D^\beta=D^{\alpha+\beta};
\end{equation*}
if $f\in L^\infty(a,b)$ or $f\in C^n[a,b]$ and $\alpha>0$, then
\begin{equation}
\label{eq:2.18}
_a^CD_x^\alpha \, {_aJ_x^\alpha} f=f
\quad \mathrm{ and }
\quad _x^CD_b^\alpha \, {_xJ_b^\alpha} f=f.
\end{equation}
On the other hand, if $f\in C^n[a,b]$ and $\alpha>0$, then
\begin{equation*}
_aJ_x^\alpha \, _a^CD_x^\alpha f(x)
=f(x)-\sum_{k=0}^{n-1}\frac{f^{(k)}(a)}{k!}(x-a)^k
\end{equation*}
and
\begin{equation*}
_xJ_b^\alpha \, _x^CD_b^\alpha f(x)
=f(x)-\sum_{k=0}^{n-1}\frac{f^{(k)}(b)}{k!}(b-x)^k.
\end{equation*}
From \eqref{eq:2.18} it is seen that the Caputo fractional derivatives
are the inverse operators for the Riemann--Liouville fractional integrals.
This is the reason why we choose here to use Caputo fractional derivatives
for solving Abel integral equations.


\section{Known Numerical Approximations to Fractional Operators}
\label{sec:3}

Diethelm \cite{bib:07} (see also \cite[pp.~62--63]{bib:08})
uses the product trapezoidal rule with respect to the weight function
$(t_k-\cdot)^{\alpha-1}$ to approximate Riemann--Liouville fractional integrals. More precisely,
the approximation
\begin{equation*}
\int_{t_0}^{t_k}(t_k-u)^{\alpha-1}f(u)du\simeq\int_{t_0}^{t_k}(t_k-u)^{\alpha-1}\widetilde{f_k}(u)du,
\end{equation*}
where $\widetilde{f_k}$ is the piecewise linear interpolator for $f$ whose nodes are chosen at
$t_j=jh$, $j=0,1,\ldots,n$ and $h=\frac{b-a}{n}$, is considered.
Odibat \cite{bib:20,bib:21} uses a modified trapezoidal rule to approximate the fractional integral
${_0J_x^\alpha}f(x)$ (Theorem~\ref{thm:apprx:FI}) and the Caputo fractional
derivative $_a^CD_x^\alpha f(x)$ (Theorem~\ref{thm:apprx:Cap}) of order $\alpha>0$.

\begin{theorem}[See \cite{bib:20,bib:21}]
\label{thm:apprx:FI}
Let $b>0$, $\alpha>0$, and suppose that the interval
$[0,b]$ is subdivided into $k$ subintervals
$[x_j,x_{j+1}]$, $j = 0,\ldots,k-1$,
of equal distances $h=\frac bk$ by using the nodes $x_j=jh$, $j=0,1,\ldots,k$.
Then the modified trapezoidal rule
\begin{multline*}
T(f,h,\alpha)=\frac{h^\alpha}{\Gamma(\alpha+2)}((k-1)^{\alpha+1}-(k-\alpha-1)k^\alpha)f(0)+f(b))\\
+\frac{h^\alpha}{\Gamma(\alpha+2)}\sum_{j=1}^{k-1}\left((k-j+1)^{\alpha+1}
-2(k-j)^{\alpha+1}+(k-j-1)^{\alpha+1}\right)f(x_j)
\end{multline*}
is an approximation to the fractional integral ${_0J_x^\alpha} f|_{x=b}$:
\begin{equation*}
{_0J_x^\alpha} f|_{x=b}=T(f,h,\alpha)-E_T(f,h,\alpha).
\end{equation*}
Furthermore, if $f\in C^2[0,b]$, then
\begin{equation*}
|E_T(f,h,\alpha)|\leq c'_{\alpha}\|f''\|_\infty b^\alpha h^2=O(h^2),
\end{equation*}
where $c'_{\alpha}$ is a constant depending only on $\alpha$.
\end{theorem}

The following theorem gives an algorithm to approximate the Caputo fractional derivative
of an arbitrary order $\alpha>0$.

\begin{theorem}[See \cite{bib:20,bib:21}]
\label{thm:apprx:Cap}
Let $b>0$, $\alpha>0$ with $n-1 < \alpha \le n$, $n \in \mathbb{N}$,
and suppose that the interval $[0,b]$
is subdivided into $k$ subintervals $[x_j,x_{j+1}]$, $j = 0,\ldots,k-1$,
of equal distances $h=\frac bk$ by using the nodes $x_j=jh$, $j=0,1,\ldots,k$.
Then the modified trapezoidal rule
\begin{multline*}
C(f,h,\alpha)=\frac{h^{n-\alpha}}{\Gamma(n+2-\alpha)}
\Biggl[\left((k-1)^{n-\alpha+1}-(k-n+\alpha-1)k^{n-\alpha}\right)
f^{(n)}(0)\\
+f^{(n)}(b)
+\sum_{j=1}^{k-1}\left((k-j+1)^{n-\alpha+1}-2(k-j)^{n-\alpha+1}
+(k-j-1)^{n-\alpha+1}\right)f^{(n)}(x_j)\Biggr]
\end{multline*}
is an approximation to the Caputo fractional derivative ${_0^CD_x^\alpha} f|_{x=b}$:
\begin{equation*}
{_0^CD_x^\alpha} f|_{x=b}=C(f,h,\alpha)-E_C(f,h,\alpha).
\end{equation*}
Furthermore, if $f\in C^{n+2}[0,b]$, then
\begin{equation*}
|E_C(f,h,\alpha)|\leq c'_{n-\alpha}\|f^{(n+2)}\|_\infty b^{n-\alpha}h^2=O(h^2),
\end{equation*}
where $c'_{n-\alpha}$ is a constant depending only on $\alpha$.
\end{theorem}

In the next section we use Theorem~\ref{thm:apprx:Cap} to find
an approximation solution to a generalized Abel integral equation.
The reader interested in other useful approximations for fractional
operators is referred to \cite{MyID:225,MyID:244,MyID:259}
and references therein.


\section{Main Results}
\label{sec:4}

Consider the following Abel integral equation of first kind:
\begin{equation}
\label{eq:4.1}
f(x)=\int_0^x\frac{g(t)}{(x-t)^{\alpha}}dt,
\quad 0<\alpha<1,
\quad 0\leq x \leq b,
\end{equation}
where $f\in C^1[a,b]$ is a given function satisfying $f(0)=0$
and $g$ is the unknown function we are looking for.

\begin{theorem}
\label{thm:mr:anal}
The solution to problem \eqref{eq:4.1} is
\begin{equation}
\label{eq:exact:sol}
g(x)=\frac{\sin(\alpha\pi)}{\pi}\int_0^x\frac{f'(t)}{(x-t)^{1-\alpha}}dt.
\end{equation}
\end{theorem}

\begin{proof}
According to Definition~\ref{def:RL:int},
we can write \eqref{eq:4.1} in the equivalent form
\begin{equation*}
f(x)=\Gamma{(1-\alpha)} \, {_0J_x^{1-\alpha}} g(x).
\end{equation*}
Then, using \eqref{eq:2.18}, it follows that
\begin{equation}
\label{eq:abel:CaputoForm}
_0^CD_x^{1-\alpha} f(x)=\Gamma{(1-\alpha)}g(x).
\end{equation}
Therefore,
$$
g(x)=\frac{1}{\Gamma(\alpha)\Gamma(1-\alpha)}\int_0^x\frac{f'(t)}{(x-t)^{1-\alpha}}dt
=\frac{\sin(\alpha\pi)}{\pi}\int_0^x\frac{f'(t)}{(x-t)^{1-\alpha}}dt,
$$
where we have used the identity
$\pi = \sin(\alpha\pi)\Gamma(\alpha)\Gamma(1-\alpha)$.
\end{proof}

Our next theorem gives an algorithm to approximate the solution \eqref{eq:exact:sol}
to problem \eqref{eq:4.1}.

\begin{theorem}
\label{thm:mr}
Let $0<x<b$ and suppose that the interval $[0,x]$ is subdivided into $k$
subintervals $[t_j,t_{j+1}]$, $j=0,\ldots,k-1$, of equal length $h=\frac{x}{k}$
by using the nodes $t_j=jh$, $j=0,\ldots,k$. An approximate solution
$\widetilde{g}$ to the solution $g$ of the Abel integral equation \eqref{eq:4.1}
is given by
\begin{multline}
\label{eq:4.12}
\widetilde{g}(x)
=\frac{h^{\alpha}}{\Gamma(1-\alpha)\Gamma(2+\alpha)}
\Biggl[\left((k-1)^{1+\alpha}-(k-1-\alpha)k^{\alpha}\right)
f'(0)+f'(x)\\
+\sum_{j=1}^{k-1}\left((k-j+1)^{1+\alpha}-2(k-j)^{1+\alpha}
+(k-j-1)^{1+\alpha}\right)f'(t_j)\Biggr].
\end{multline}
Moreover, if $f\in C^3[0,x]$, then
$g(x)=\widetilde{g}(x)-\frac{1}{\Gamma(1-\alpha)}E(x)$
with
\begin{equation*}
|E(x)|\leq c'_{\alpha}\|f'''\|_\infty x^{\alpha} h^2 = O(h^2),
\end{equation*}
where $c'_{\alpha}$ is a constant depending only on $\alpha$
and $\|f'''\|_\infty =\max_{t\in[0,x]} |f'''(t)|$.
\end{theorem}

\begin{proof}
We want to approximate the Caputo derivative in \eqref{eq:abel:CaputoForm},
i.e., to approximate
\begin{equation*}
g(x)=\frac{1}{\Gamma{(1-\alpha)}} {_0^CD_x^{1-\alpha}} f(x)
= \frac{\sin(\alpha \pi) \, \Gamma(\alpha)}{\pi} {_0^CD_x^{1-\alpha}} f(x).
\end{equation*}
If the Caputo fractional derivative of order $1-\alpha$ for $f$, $0<\alpha<1$,
is calculated at collocation nodes $t_j$, $j=0,\ldots,k$,
then the result is a direct consequence of Theorem~\ref{thm:apprx:Cap}.
\end{proof}


\section{Illustrative Examples}
\label{sec:5}

We exemplify the approximation method given by Theorem~\ref{thm:mr}
with three Abel integral equations whose exact solutions are found
from Theorem~\ref{thm:mr:anal}. The computations were done with
the Computer Algebra System \textsf{Maple}. The complete code is provided
in Appendix~\ref{app}.

\begin{example}
\label{Example1}
Consider the Abel integral equation
\begin{equation}
\label{eq:prb1}
e^x-1=\int_0^x\frac{g(t)}{(x-t)^{1/2}}dt.
\end{equation}
The exact solution to \eqref{eq:prb1} is given by Theorem~\ref{thm:mr:anal}:
\begin{equation}
\label{eq:es:ex1}
g(x)=\frac{e^x}{\sqrt{\pi}} {erf}(\sqrt{x}),
\end{equation}
where $erf(x)$ is the error function, that is,
$$
erf(x)=\frac{2}{\sqrt{\pi}}\int_ 0^x e^{-t^2}dt.
$$
In Table~\ref{Table4.1} we present the approximate values for $g(x_i)$,
$x_i=0.1,0.2,0.3$, obtained from Theorem~\ref{thm:mr} with $k=1$, $10$, $100$.
\begin{table}[ht]
\caption{Approximated values $\widetilde{g}(x_i)$ \eqref{eq:4.12}
to the solution \eqref{eq:es:ex1} of \eqref{eq:prb1}
obtained from Theorem~\ref{thm:mr} with $k=1$, $10$, $100$.}
\label{Table4.1}
\begin{tabular}{|c|c|c|c|c|c|}
\hline
$x_i$ & $k=1$ & $k=10$ & $k=100$ & Exact solution \eqref{eq:es:ex1} & Error $\Delta_{k=100}$\\\hline
$0.1$   & $0.2154319668$ & $0.2152921762$  & $0.2152904646$   & $0.2152905021$ & $3.75 \times 10^{-8}$\\
$0.2$   & $0.3267280013$ & $0.3258941876$  & $0.3258841023$   & $0.3258840762$ & $2.61 \times 10^{-8}$\\
$0.3$   & $0.4300194238$ & $0.4275954299$  & $0.4275658716$   & $0.4275656575$ & $2.14 \times 10^{-7}$\\\hline
\end{tabular}
\end{table}
\end{example}

\begin{example}
\label{Example2}
Consider the following Abel integral equation of first kind:
\begin{equation}
\label{eq:prb2}
x=\int_0^x\frac{g(t)}{(x-t)^{4/5}}dt.
\end{equation}
The exact solution \eqref{eq:exact:sol} to \eqref{eq:prb2} takes the form
\begin{equation}
\label{eq:es:ex2}
g(x)=\frac{5}{4}\,\frac{\sin\left(\frac{\pi}{5}\right)}{\pi}\,{x}^{4/5}.
\end{equation}
We can see the numerical approximations of $g(x_i)$,
$x_i=0.4,0.5,0.6$, obtained from Theorem~\ref{thm:mr} with
$k=1$, $10$  in Table~\ref{Table4.2}.
\begin{table}[ht]
\caption{Approximated values $\widetilde{g}(x_i)$ \eqref{eq:4.12}
to the solution \eqref{eq:es:ex2} of \eqref{eq:prb2}
obtained from Theorem~\ref{thm:mr} with $k=1$, $10$.}
\label{Table4.2}
\begin{tabular}{|c|c|c|c|c|}
\hline
$x_i$ & $k=1$  & $k=10$ & Exact solution \eqref{eq:es:ex2} & Error $\Delta_{k=10}$\\\hline
$0.4$   & $0.1123639036$ & $0.1123639036$ & $0.1123639037$ & $1\times 10^{-10}$\\
$0.5$   & $0.1343243751$ & $0.1343243751$ & $0.1343243752$ & $1\times 10^{-10}$\\
$0.6$   & $0.1554174667$ & $0.1554174668$ & $0.1554174668$ & $\le 10^{-11}$ \\\hline
\end{tabular}
\end{table}
\end{example}

\begin{example}
\label{Example3}
Consider now the Abel integral equation
\begin{equation}
\label{eq:prb3}
x^{7/6}=\int_0^x\frac{g(t)}{(x-t)^{1/3}}dt.
\end{equation}
Theorem~\ref{eq:exact:sol} asserts that
\begin{equation}
\label{eq:es:ex3}
g(x)= \frac{7\sqrt{3}}{12\,\pi}\,\displaystyle \int_{0}^{x}\!
{\frac{t^{1/6}}{ \left( x-t\right)^{2/3}}}{dt}.
\end{equation}
The numerical approximations of $g(x_i)$, $x_i=0.6,0.7,0.8$,
obtained from Theorem~\ref{thm:mr} with $k=10, 100, 1000$,
are given in Table~\ref{Table4.3}.
\begin{table}[ht]
\caption{Approximated values $\widetilde{g}(x_i)$ \eqref{eq:4.12}
to the solution \eqref{eq:es:ex3} of \eqref{eq:prb3}
obtained from Theorem~\ref{thm:mr} with $k=10, 100, 1000$.}
\label{Table4.3}
\begin{tabular}{|c|c|c|c|c|c|}
\hline
$x_i$ & $k=10$ & $k=100$ & $k=1000$ & Exact solution \eqref{eq:es:ex3} & Error $\Delta_{k=1000}$\\\hline
$0.6$ & $0.6921182258$ & $0.6981839386$  & $0.6985886509$  & $0.6986144912$ & $0.0000258403$ \\
$0.7$ & $0.7475731262$ & $0.7541248475$  & $0.7545620072$  & $0.7545898940$ & $0.0000278868$\\
$0.8$ & $0.7991892838$ & $0.8061933689$  & $0.8066606797$  & $0.8066905286$ & $0.0000298489$\\\hline
\end{tabular}
\end{table}
\end{example}


\appendix


\section{Maple Code for Examples of Section~\ref{sec:5}}
\label{app}

We provide here all the definitions and computations done in \textsf{Maple} for
the problems considered in Section~\ref{sec:5}. The definitions
follow closely the notations introduced along the paper, and should be clear
even for readers not familiar with the Computer Algebra System \textsf{Maple}.

\bigskip

\small

\noindent \verb!> # Solution given by Theorem!~\ref{thm:mr:anal}
\begin{verbatim}
> g := (f, alpha, x) -> sin(alpha*Pi)*(int((diff(f(t), t))/(x-t)^(1-alpha),
                                            t = 0..x))/Pi:
\end{verbatim}
\noindent \verb!> # Approximation given by Theorem!~\ref{thm:mr}
\begin{verbatim}
> gtilde := (f, alpha, h, k, x) -> h^alpha
             *(((k-1)^(1+alpha)-(k-1-alpha)*k^alpha)
             *(D(f))(0)+(D(f))(x)+sum(((k-j+1)^(1+alpha)-2*(k-j)^(1+alpha)
             +(k-j-1)^(1+alpha))*(D(f))(j*h), j = 1 .. k-1))
             /(GAMMA(1-alpha)*GAMMA(2+alpha)):
\end{verbatim}
\noindent \verb!> # Example!~\ref{Example1}
\begin{verbatim}
> f1 := x -> exp(x)-1:
> g(f1, 1/2, x);
\end{verbatim}
$$
{\frac {{{\rm erf}\left(\sqrt {x}\right)}{{\rm e}^{x}}}{\sqrt {\pi }}}
$$
\begin{verbatim}
> ExactValues1 := evalf([g(f1,1/2,0.1), g(f1,1/2,0.2), g(f1,1/2,0.3)]);
\end{verbatim}
$$
           [0.2152905021, 0.3258840762, 0.4275656575]
$$
\begin{verbatim}
> ApproximateValues1 := k -> evalf([gtilde(f1, 1/2, 0.1/k, k, 0.1),
                                    gtilde(f1, 1/2, 0.2/k, k, 0.2),
                                    gtilde(f1, 1/2, 0.3/k, k, 0.3)]):
> ApproximateValues1(1);
\end{verbatim}
$$
           [0.2154319668, 0.3267280013, 0.4300194238]
$$
\begin{verbatim}
> ApproximateValues1(10);
\end{verbatim}
$$
           [0.2152921762, 0.3258941876, 0.4275954299]
$$
\begin{verbatim}
> ApproximateValues1(100);
\end{verbatim}
$$
           [0.2152904646, 0.3258841023, 0.4275658716]
$$
\noindent \verb!> # Example!~\ref{Example2}
\begin{verbatim}
> f2 := x -> x:
> g(f2, 4/5, x);
\end{verbatim}
$$
\frac{5}{4}\,{\frac {\sin \left( \frac{1}{5}\,\pi  \right) {x}^{4/5}}{\pi }}
$$
\begin{verbatim}
> ExactValues2 := evalf([g(f2,4/5,0.4), g(f2, 4/5, 0.5), g(f2, 4/5, 0.6)]);
\end{verbatim}
$$
           [0.1123639037, 0.1343243752, 0.1554174668]
$$
\begin{verbatim}
> ApproximateValues2 := k -> evalf([gtilde(f2, 4/5, 0.4/k, k, 0.4),
                                    gtilde(f2, 4/5, 0.5/k, k, 0.5),
                                    gtilde(f2, 4/5, 0.6/k, k, 0.6)]):
> ApproximateValues2(1);
\end{verbatim}
$$
           [0.1123639036, 0.1343243751, 0.1554174667]
$$
\begin{verbatim}
> ApproximateValues2(10);
\end{verbatim}
$$
           [0.1123639036, 0.1343243751, 0.1554174668]
$$
\noindent \verb!> # Example!~\ref{Example3}
\begin{verbatim}
> f3 := x -> x^(7/6):
> g(f3, 1/3, x);
\end{verbatim}
$$
\frac{1}{2}\,\frac{\sqrt {3}
\displaystyle \int_{0}^{x}\!\frac{7}{6}\,{\frac{t^{1/6}}{ \left( x-t
 \right)^{2/3}}}{dt}}{\pi}
$$
\begin{verbatim}
> ExactValues3 := evalf([g(f3,1/3,0.6), g(f3,1/3,0.7), g(f3,1/3,0.8)]);
\end{verbatim}
$$
           [0.6986144912, 0.7545898940, 0.8066905286]
$$
\begin{verbatim}
> ApproximateValues3 := k -> evalf([gtilde(f3, 1/3, 0.6/k, k, 0.6),
                                    gtilde(f3, 1/3, 0.7/k, k, 0.7),
                                    gtilde(f3, 1/3, 0.8/k, k, 0.8)]):
> ApproximateValues3(1);
\end{verbatim}
$$
           [0.5605130027, 0.6054232384, 0.6472246667]
$$
\begin{verbatim}
> ApproximateValues3(10);
\end{verbatim}
$$
           [0.6921182258, 0.7475731262, 0.7991892838]
$$
\begin{verbatim}
> ApproximateValues3(100);
\end{verbatim}
$$
           [0.6981839386, 0.7541248475, 0.8061933689]
$$
\begin{verbatim}
> ApproximateValues3(1000);
\end{verbatim}
$$
           [0.6985886509, 0.7545620072, 0.8066606797]
$$

\normalsize


\section*{Acknowledgements}

This work is part of first author's PhD project.
Partially supported by Azad University, Iran;
and CIDMA--FCT, Portugal, within project
PEst-OE/MAT/UI4106/ 2014. 
Jahanshahi was supported by a scholarship from the Ministry of Science,
Research and Technology of the Islamic Republic of Iran,
to visit the University of Aveiro.
The hospitality and the excellent working conditions
at the University of Aveiro are here gratefully acknowledged.
The authors would like also to thank
three anonymous referees for valuable comments.



\end{document}